\documentclass{amsart}
\usepackage{amsthm}
\usepackage{amscd}
\usepackage{amsfonts}
\usepackage{amstext}
\usepackage{amssymb}
\usepackage{amsmath}
\usepackage{amsxtra}

\usepackage{amsmath,amsthm,amsfonts,latexsym,amscd,amssymb,enumerate}

\usepackage[
]{hyperref}

\usepackage{colordvi}

\newtheorem{prop}{Proposition}[section]

\newtheorem{lem}[prop]{{Lemma}}

\newtheorem{proposition}[prop]{Proposition}
\newtheorem{theorem}[prop]{Theorem}
\newtheorem{corollary}[prop]{{Corollary}}
\newtheorem{lemma}[prop]{Lemma}

\theoremstyle{definition}

\newtheorem{remark}[prop]{Remark}
\newtheorem{definition}[prop]{Definition}
\newtheorem{example}[prop]{Example}
\newtheorem{question}[prop]{Question}
\newtheorem{conjecture}[prop]{Conjecture}

\numberwithin{equation}{section}

\newcommand{\Z}{\mathbb{Z} }

\newcommand{\C}{\mathbb{C} }

\newcommand{\complexs}{\mathbb{C}}
\newcommand{\reals}{\mathbb{R}}
\newcommand{\integers}{\mathbb{Z}}
\newcommand{\naturals}{\mathbb{N}}

\newcommand{\into}{\hookrightarrow}
\newcommand{\tensor}{\otimes}
\newcommand{\iso}{\cong}
\DeclareMathOperator{\supp}{supp}

\newcommand{\abs}[1]{\left\lvert#1\right\rvert} 
\newcommand{\innerprod}[1]{\langle #1 \rangle}

\newcommand{\ind}{\operatorname{ind}}

\begin{document}

\title{Large scale index of multi-partitioned manifolds}
\author{Thomas Schick \and Mostafa Esfahani Zadeh}
\address{Thomas Schick, 
Mathematisches Institut,
Georg-August-Universit\"at G\"ottingen,
Germany
\newline
Mostafa Esfahani Zadeh,
Sharif University of Technology,
Tehran-IRAN}
\email{thomas.schick@math.uni-goettingen.de}
\email{esfahani@sharif.ir}

\begin{abstract}

Let M be a complete n-dimensional Riemannian spin manifold, partitioned by q two-sided
hypersurfaces  which have a compact transverse intersection
N and which in addition satisfy a certain coarse transversality
condition. Let E be a Hermitean bundle on M with connection.
We define a coarse multi-partitioned index of the spin Dirac
operator on M twisted by E.

Our main result is the computation of this multi-partitioned index as
 the Fredholm index of the Dirac operator
on the compact manifold N, twisted by the restriction of E to N.

We establish the following main application: if the scalar curvature of M is
bounded below by a positive constant everywhere (or even if this happens only on one of the
quadrants defined by the partitioning hypersurfaces) then the
multi-partitioned index vanishes. Consequently, the multi-partitioned index
is an obstruction to uniformly positive scalar curvature on M.

The proof of the multi-partitioned index theorem proceeds in two steps: first
we establish a strong new localization property of the multi-partitioned
index which is the main novelty of this note. This we establish even when we
twist with an arbitrary Hilbert A-module bundle E (for an auxiliary
C*-algebra A).
This allows to reduce to the case where M is the product of N with Euclidean
space of dimension q. For this
special case, standard methods for the explicit calculation of the index in
this product situation can be adapted to obtain the result.
 \end{abstract}



\maketitle
\section{Introduction}

Consider a complete Riemannian spin manifold $(M,g)$. If $M$ is
non-compact, which is the situation we are interested in, the classical index
theory of elliptic operators (like the Dirac operator) usually can not be
applied because of lack of the Fredholm property. 

In this situation, using non-commutative geometry and operator algebras, John 
Roe has initiated an adapted index theory which we call ``\emph{large scale
  index theory}'' or sometimes ``\emph{coarse index theory}'' (compare
e.g.~\cite{Roe-index}). The main
player is the \emph{Roe algebra} $C^*(X)$, associated to a complete proper
metric space $X$. This is an algebra of operators (acting on function spaces on $X$),
mildly depending on choices. Its K-theory $K_*(C^*(X))$ is canonically and
functorially associated to $X$, with functoriality for proper and uniformly
expansive maps.

If $M$ is a complete Riemannian spin manifold of dimension $n$ and $E\to M$ is
a Hermitian bundle with connection, we have
the twisted Dirac operator $D_E$. More generally, if $A$ is an auxiliary
$C^*$-algebra and $E$ is a Hilbert $A$-module bundle with $A$-linear
connection one can form $D_E$ as regular unbounded operator in the Hilbert
$A$-module sense. One of the main virtues of large scale index
theory is the construction of its \emph{coarse index} $\ind_c(D_E)\in
K_n(C^*(M;A))$. This index contains information about the geometry: on the one
hand it depends on the metric  only up to bilipschitz equivalence. On the other
hand, it vanishes if the metric has uniformly positive scalar curvature and
$E$ is flat
(actually, it suffices to have this property on sufficiently large subsets of
$M$). It is therefore important to get information about this coarse index. In
the present paper, we will do this for multi-partitioned manifolds.

\begin{definition}\label{def:multipart}
  A complete  $n$-dimensional Riemannian manifold $M$ is
  \emph{multi-partitioned} by codimension $1$ hypersurfaces $M_1,\dots,M_q$ if
  \begin{itemize}
  \item each of the $M_k$ is a two-sided, separating $M=M_k^+\cup M_k^-$ with
    $M_k^+\cap M_K^-=M_k$;
  \item $N:=\bigcap_{k=1}^q M_k$ is compact, and in a neighborhood of $N$
   the $M_i$ intersect mutually transversally. In particular, $N$ is itself a
    submanifold of dimension $n-q$ with trivial normal bundle;
  \item the collection of hypersurfaces is \emph{coarsely transversal} in the
    sense that for each $r>0$ there is $s>0$ such that $\bigcap_{k=1}^q
    U_r(M_k)\subset U_s(N)$. Here, for a subset $X\subset M$ we set
    $U_r(X):=\{p\in M\mid d(p,X)\le r\}$, the $r$-neighborhood of $X$.
  \end{itemize}
\end{definition}

In Section \ref{sec:largescaleind} we will prove
\begin{lemma}\label{lem:map_to_Rq}
  If $M$ is a multi-partitioned manifold, partitioned by $M_1,\dots, M_q$, the
  signed distance to the $M_k$ 
  defines a proper and uniformly expansive map $f\colon M\to \reals^q$ which
  is smooth near $N:=f^{-1}(0)$ and such that $0$ is a regular value.
\end{lemma}
The notion of uniform expansiveness is recalled in Definition
\ref{def:functor}. Also the other notions of coarse index theory used in this
introduction are recalled in Section \ref{section2}.

\begin{definition}
  Let $M$ be a multi-partitioned spin manifold and $E\to M$ a Hilbert
  $A$-module bundle 
  with connection. We define the multi-partitioned index of
  the Dirac operator twisted by $E$ as
  \begin{equation*}
    \ind_p(D_E) :=\kappa(f_*(\ind_c(D_E))) \in K_{n-q}(C^*(\reals^n;A)).
  \end{equation*}
  Here, $f_*\colon K_*(C^*(M;A))\to K_*(C^*(\reals^q;A))$ is obtained by
  functoriality of the K-theory of the twisted Roe algebra and
  \begin{equation*}
    \kappa\colon K_*(C^*(\reals^q;A))\to K_{*-q}(A)
  \end{equation*}
  is a canonical isomorphism we recall in Corollary \ref{corol:KRq}. 
\end{definition}

Our main result is the calculation of this multi-partitioned manifold index:

\begin{theorem}\label{theo:part_mf}
  Let $M$ be a complete spin manifold of dimension $n$ with proper continuous
  uniformly expansive map 
  $f\colon M\to \reals^q$. Assume that $f$ is smooth near $N:=f^{-1}(0)$ and
  that $0$ is a regular value. This implies that $N$ is a compact submanifold
  with trivial normal bundle of dimension $n-q$. 
  In particular, $N$ inherits a
  spin structure from $M$. For example, $M$ could be multi-partitioned by
  $M_1,\dots, M_q$. Let $E\to M$ be a Hermitean bundle with
  connection. Then
  \begin{equation*}
   \kappa( f_*(\ind_c(D_{M,E}))) = \ind(D_{N,E|_N}) \in K_{n-q}(\complexs).
  \end{equation*}
 Note
  that $K_{n-q}(\complexs)=\integers$ if $n-q$ is even, and
  $K_{n-q}(\complexs)=\{0\}$ if $n-q$ is odd and $\ind(D_{N,E|_N})\in
  K_{n-q}(\C)$ is the index of
  the Dirac operator on the compact spin manifold $N$, twisted by the
  Hermitean bundle $E|_N$.
\end{theorem}

\begin{remark}
  We strongly expect that Theorem \ref{theo:part_mf} generalizes to arbitrary
  $C^*$-algebras $A$ and Hilbert $A$-modules $E$ over $M$,
  equating the multi-partitioned index with the Mishchenko-Fomenko index
  $\ind(D_{N,E|_N})\in K_{n-q}(A)$ 
  of the Dirac operator on the compact manifold $N$, twisted with the Hilbert
  $A$-module bundle $E|_N$. We comment more on this in Section
  \ref{sec:generalization}. 
\end{remark}

Historically the first version of Theorem \ref{theo:part_mf} is the
\emph{partitioned manifold index 
  theorem} of Roe and Higson, for the case $q=1$ with
several proofs, e.g.~in \cite{Roe-index} or \cite{Higson-cobordism}. Here only
the case 
$n$ odd is interesting (and treated), as otherwise the target group
$K_1(\complexs)=0$. In \cite{Zadeh1,Zadeh2}, the approach of
\cite{Higson-cobordism} is generalized, still for $q=1$, from $\complexs$ to
arbitrary 
coefficient algebras $A$, as long as $n$ is odd. Finally, Siegel treats the
case of multi-partitioned manifolds with
additional geometric restrictions in \cite{P.Siegel}. 

All these proofs consist of two steps. The first is a reduction to the product
case $M=N\times \reals^q$, and the second is a more or less explicit
calculation in this product case.

In this note we develop a new and particularly strong method for the reduction
step. Indeed, we prove in particular the following:
\begin{proposition}\label{prop:loc_ind}
  Assume that $f\colon M\to \reals^q$ with Hilbert $A$-module bundle $E\to M$
  with connection
  and $f'\colon M'\to \reals^q$ with Hilbert $A$-module bundle $E'\to M'$ with
  connection are
  two 
  complete Riemannian spin manifolds as in Theorem \ref{theo:part_mf}. Assume
  there are open neighborhoods $U$ of $f^{-1}(0)$ in $M$ and $U'$ of
  $f^{-1}(0)$ in $M'$ and a spin-structure preserving isometry $\psi\colon
  U\to U'$ which is covered by an $A$-isometry $\Psi\colon E|_U\to
  E'|_{U'}$ preserving the connections. Then 
  \[f_*(\ind_c(D_E)) = f'_*(\ind_c(D_{E'}))\in K_n(C^*(\reals^q;A)).\]
\end{proposition}

Indeed, Proposition \ref{prop:loc_ind} is a corollary of a localization
theorem for classes in $K_*(C^*(\reals^q;A))$: if two such are obtained as
indices of operators which coincide on an arbitrary non-empty open subset of
$\reals^q$, then they are already equal.

An consequence of theorem \ref{theo:part_mf} is the following obstruction to
positive scalar curvature.

\begin{theorem}\label{theo:obstructmet}
Let $(M,g)$ be a complete Riemannian spin manifold, partitioned transversally
and coarsely by $q$
hypersurfaces whose intersection is a compact manifold $N$. If $\hat A(N)\neq
0$ then  
the scalar curvature of $g$ (or of any other complete Riemannian metric which
is  bilipschitz equivalent to $g$)
can not be uniformly positive outside a compact subset of any quadrant formed
by the partitioning hypersurfaces.  
\end{theorem}

\section{Basics of large scale index theory}\label{section2}

We start with a very brief review of the Roe algebra $C^*(X;A)$ and a companion,
the structure algebra $D^*(X;A)$, inside which $C^*(X;A)$ is an ideal. For
simplicity, we assume that $X$ is a positive dimensional Riemannian manifold
throughout (possibly with boundary). We follow
\cite{Roe-index,Piazza-Schick}. For the basics of Hilbert $A$-modules and
their operators (adjointability, $A$-compactness,\ldots), compare
\cite{Lance}.

\begin{definition}
  Given a Hilbert $A$-module bundle $E\to X$, consider the Hilbert $A$-module
  $L^2(E)$ of square integrable sections of $E$
  with $A$-valued inner product given by integration of the pointwise
  $A$-valued inner product. 

  One defines the \emph{structure algebra} $D^*(X;A)$ as $C^*$-closure of the
  algebra of bounded adjointable $A$-linear operators $T$ on $L^2(E)$ which
  satisfy 
  \begin{enumerate}
  \item $T$ has finite propagation, i.e.~there is $R>0$ such that $\supp(T
    s)\subset U_R(\supp(s))$ for each $s\in L^2(E)$.
  \item $T$ is pseudolocal: for any compactly supported
    continuous functions $\phi$ the operator $\phi
    T- T\phi$ is an $A$-compact operator, where we let $\phi$ act on $L^2(E)$ as
    multiplication operator.
  \end{enumerate}

  The \emph{Roe algebra} $C^*(X;A)$ is the norm closure of operators
  $T$ as above which satisfy
  \begin{enumerate}
  \item $T$ has finite propagation
  \item $T$ is \emph{locally compact}, i.e.~for every compactly supported continuous function
    $\phi$, $T\phi$ and $\phi T$ are $A$-compact operators.
  \end{enumerate}
  One checks immediately that $C^*(X;A)$ is an ideal in $D^*(X;A)$.
\end{definition}

\begin{remark}
  Strictly speaking, one has to enlarge $E$ by tensoring with
  $l^2(\naturals)$; we gloss over this details as all we have to do happens in
  a fixed summand canonically isomorphic to $E$, as in \cite{HankePapeSchick}.
\end{remark}

\begin{definition}\label{def:functor}
  Let $f\colon X\to Y$ be a continuous map between complete Riemannian
  manifolds.
  \begin{itemize}
  \item $f$ is called proper if the inverse image of every compact subset of
    $Y$ is compact
  \item $f$ is uniformly expansive if for every $r>0$ there is $s>0$ such that
    $d(f(x),f(y))\le s$ whenever $d(x,y)\le r$.
  \end{itemize}

  Given, in addition, Hilbert $A$-module bundles $E\to X$ and $F\to Y$, an 
  isometric $A$-embedding $V\colon L^2(E)\to L^2(F)$
  is said to \emph{cover $f$ in the $D^*$-sense} if $V$ is a norm-limit of
  operators $V$ such that
  \begin{itemize}
  \item $V$ has \emph{finite propagation}, i.e.~there is $R>0$ such that
    $\supp(V s)\subset U_R(f(\supp(s)))$ 
  \item whenever $\phi$ is a compactly supported continuous function on $X$
    and $\psi$ is a compactly supported continuous function on $Y$ such that
    $\phi\cdot( \psi\circ f)=0$ then $\psi V \phi$ is compact.
  \end{itemize}

By \cite[Lemma 7.7]{HigsonRoeCBC} one can always find an isometry $V$ covering
$f$ in the $D^*$-sense, even itself with prescribed finite propagation $R>0$. 

Given such an isometry $V$, $Ad_V(T):= VTV^*$ then defines a map $Ad_V\colon D^*(X;A)\to
D^*(Y;A)$ which restricts to a map $C^*(X;A)\to C^*(Y;A)$.

As in \cite[Lemma 3]{HigsonRoeYu}, the induced map on K-theory does not depend
on the choice of $V$, but only on $f$. This implies that
\begin{itemize}
\item $C^*(X;A)$ and $D^*(X;A)$ and therefore also the quotient algebra
  $D^*(X;A)/C^*(X;A)$ are well defined up to non-canonical isomorphism: for
  the different choices of the Hilbert module $E$ entering the definition (but
  which we avoided in the notation) we find non-canonical isomorphisms of the
  resulting $C^*$-algebras.
\item  $K_*(C^*(X;A))$, $K_*(D^*(X;A))$, and $K_*(D^*(X;A)/C^*(X;A))$ are well
  defined up to canonical isomorphism and are functorial for proper continuous
  uniformly expansive maps.
\end{itemize}

One defines $K_*(X;A):= K_{*+1}(D^*(X;A)/C^*(X;A))$, the \emph{locally finite
  K-homology of $X$ with coefficients in $A$}.
\end{definition}

In the original papers of Roe the compactness condition on $V$ was forgotten
to be 
mentioned. It is introduced (with this terminology) in \cite[Definition
1.7]{Piazza-Schick} 
or (under the name ``covers topologically'') in \cite[Definition
2.4]{P.Siegel}.


We will use vanishing of these K-groups in a number of
situations. The most powerful and useful concept in this context is 
\emph{flasqueness}. 

\begin{definition}
  A complete Riemannian manifold $M$ (possibly with boundary) is called
  \emph{flasque} if there is a continuous, proper and uniformly expansive 
  map $f\colon M\to M$ with the following properties:
  \begin{enumerate}
  \item there is a continuous uniformly expansive proper homotopy between $f$
    and the identity
  \item for every compact subset $K\subset M$ there is an $N\in\naturals$ such
    that $f^N(M)\cap K=\emptyset$.
  \end{enumerate}
\end{definition}

\begin{example}
  For an arbitrary manifold $X$, the product $X\times [0,\infty)$ is
  flasque. Indeed, the map $f\colon X\times [0,\infty)\to X\times [0,\infty); (x,t)\mapsto
  (x,t+1)$ satisfies the conditions required in the definition of
  flasqueness.
  \end{example}

\begin{proposition}\label{prop:flasque_vanishing}
  If $M$ is flasque, then
  \begin{equation*}
    K_*(C^*(M;A))=0;\qquad K_*(D^*(M;A))=0;\qquad K_*(M;A)=0.
  \end{equation*}
\end{proposition}
\begin{proof}
   For $A=\complexs$, this is proved in \cite[Proposition 9.4]{Roe-index}. The
   proof carries over to general $A$ almost literally.
\end{proof}


The final ingredient we will need from the basics of large scale index theory
is a Mayer-Vietoris principle. 

\begin{definition}\label{def:rel_algs}
  Let $M$ be a complete Riemannian manifold and $Y\subset M$ a closed
  subset. We define $C^*(Y\subset M;A)\subset C^*(M;A)$ as the closure of the 
  set of all 
  operators $T\in C^*(M;A)$ which have \emph{support near $Y$}, i.e.~such that
  there is
  $R>0$ such that $T\phi =0$ and $\phi T=0$ whenever $\phi\in C_{comp}(M)$
  with $\supp(\phi)\cap
  U_R(Y)=\emptyset$.

  We define $D^*(Y\subset M;A)$ as the closure of the set of all operators
  $T\in D^*(M;A)$ such that $T$ has support near $Y$ and in addition $T\phi$
  and $\phi T$ are $A$-compact operators whenever $\phi\in C_{comp}(M)$ with 
  $\supp(\phi)\cap Y=\emptyset$. Then $C^*(Y\subset M;A)$ and $D^*(Y\subset
  M;A)$ are both ideals in $D^*(M;A)$.
\end{definition}

\begin{proposition}\label{prop:rel_K}
  In the situation of Definition \ref{def:rel_algs}, the canonical maps
  \begin{equation*}
    \begin{split}
      C^*(Y;A)&\into C^*(Y\subset M;A);\qquad D^*(Y;A)\into D^*(Y\subset M;A);\\
    D^*(Y;A)/C^*(Y;A)&\to D^*(Y\subset M;A)/C^*(Y\subset M;A)
    \end{split}
  \end{equation*}
  induce isomorphism in K-theory.
\end{proposition}
\begin{proof}
  This is proved for $A=\complexs$ in \cite{P.Siegel}, the proof carries over
  almost literally to general $A$.

\end{proof}
\begin{definition}
  Assume that $M=M_1\cup M_2$ with intersection $M_0:=M_1\cap M_2$ for closed
  subsets $M_1,M_2$. This decomposition is called \emph{coarsely excisive} if
  for each $r>0$ there is $s>0$ such that $U_r(M_1)\cap U_r(M_2)\subset
  U_s(M_0)$. 
\end{definition}

\begin{theorem}\label{theo:MV}
  Assume that $M$ is a complete Riemannian manifold with a coarsely excisive
  decomposition $M=M_1\cup M_2$ into closed subset. Then we have long exact
  Mayer-Vietoris sequences
  \begin{equation*}
    \begin{split}
\dots  &\to K_j(C^*(M_1;A))\oplus K_j(C^*(M_2;A))\to K_j(C^*(M;A))\xrightarrow{\delta_{MV}}
      K_{j-1}(C^*(M_0;A))\to\dots \\
 \dots  &      \to K_j(D^*(M_1;A))\oplus K_j(D^*(M_2;A))\to K_j(D^*(M;A))\xrightarrow{\delta_{MV}}
      K_{j-1}(D^*(M_0;A))\to\dots \\
\dots   &       \to K_j(M_1;A)\oplus K_j(M_2;A)\to K_j(M;A)\xrightarrow{\delta_{MV}}
      K_{j-1}(M_0;A)\to\dots \\
    \end{split}
  \end{equation*}
  These long exact sequences are compatible with the long exact sequences in
  K-theory of the 
  extensions $0\to C^*\to D^*\to D^*/C^*\to 0$. 
\end{theorem}
\begin{proof}
  For $A=\complexs$, this is the main result of \cite[Section
  3]{P.Siegel}. The proof carries over almost literally to general $A$.
\end{proof}

\begin{corollary}\label{corol:KRq}
  For an arbitrary  complete Riemannian manifold
  $M$ there is a commuting diagram with horizontal isomorphisms
  \begin{equation*}
    \begin{CD}
      K_*(D^*(M\times \reals;A)) @>{\delta_{MV}}>{\iso}> K_{*-1}(D^*(M;A))\\
        @VVV @VVV\\
 K_*(D^*(M\times \reals;A)/C^*(M\times\reals;A)) @>{\delta_{MV}}>{\iso}> K_{*-1}(D^*(M;A)/C^*(M;A))\\
        @VVV @VVV\\
  K_{*-1}(C^*(M\times \reals;A)) @>{\delta_{MV}}>{\iso}> K_{*-2}(C^*(M;A)).
    \end{CD}
  \end{equation*}
  In particular, for arbitrary $j,q$ we obtain a canonical isomorphism
  \begin{equation*}
    \kappa \colon K_j(C^*(\reals^q;A))\to K_{j-q}(A).
  \end{equation*}
\end{corollary}
\begin{proof}
  The decomposition $M\times \reals=M\times  (-\infty,0]\cup
  M\times [0,\infty)$ is coarsely excisive, and the half spaces
  $M\times (-\infty,0]$, $M\times [0,\infty)$ are
  flasque. Combining  \ref{prop:flasque_vanishing} and Theorem
  \ref{theo:MV}, the Mayer-Vietoris boundary map gives an isomorphism
  \begin{equation*}
    \delta_{MV}\colon K_j(D^*(M\times \reals;A))\xrightarrow{\iso}
    K_{j-1}(D^*(M;A)),
  \end{equation*}
  and similarly for the other algebras.

   The $q$-fold iteration of this then gives an isomorphism
   $K_j(C^*(\reals^q;A))\to K_{j-q}(C^*(\reals^0;A))$ with a canonical
   isomorphism,  $C^*(pt;A)= A\tensor\mathbf K$, and of course $K_*(A\tensor
   \mathbf K)\iso K_*(A)$.
\end{proof}

Finally, we recall from \cite{Roe-index}, \cite{HankePapeSchick} how to define
the coarse index of a twisted Dirac 
operator and we mention its main properties.
Assume therefore that $M$ is a complete spin manifold
(without boundary) and $E\to M$ a Hilbert $A$-module bundle with
connection. The twisted Dirac operator $D_E$ with its natural domain is then a
self-adjoint unbounded operator on the Hilbert $A$-module
$L^2(S\tensor E)$ where $S$ is the spinor bundle of $M$. If we assume
that $M$ is even dimensional, $S=S^+\oplus S^-$ is $\integers/2$-graded and
$D_E$ is an odd operator with respect to this grading.
Let $\chi\colon \reals\to [-1,1]$ an odd continuous function
such that $\chi(t)\xrightarrow{t\to\pm\infty} \pm 1$. Then $\chi(D_E)$ is still
an odd operator. The Fourier inversion formula and unit propagation
of the wave operator imply that $\chi(D_E)\in D^*(M;A)$
and that $\chi(D_E)^2-1\in C^*(M;A)$, compare \cite{HankePapeSchick}.

\begin{definition}
  If $\dim(M)$ is even, choose a measurable fiberwise isometry $S^+\to S^-$ with induced $A$-linear
  isometry of 
  propagation zero $V\colon L^2(S^+\tensor E)\to L^2(S^-\tensor E)$.

  Because $\chi(D_E)^2-1\in C^*(M;A)$, the operator $V^*\chi(D_E)_+\colon
  L^2(S_+\tensor E)\to L^2(S_+\tensor E)$ 
  is a unitary operator in
  $D^*(M;A)/C^*(M;A)$ and therefore represents a class $[D_E] \in
  K_1(D^*(M;A)/C^*(M;A)) = K_1(M;A)$, the \emph{fundamental K-homology class}.

  One has the long exact sequence in K-theory associated to the extension
  \begin{equation*}
  0\to C^*(M;A)\to D^*(M;A)\to D^*(M;A)/C^*(M;A)\to 0
\end{equation*}
with boundary map
  $\partial\colon K_1(D^*(M;A)/C^*(M;A))\to K_0(C^*(M;A))$. The \emph{large
    scale index} is defined to be 
  $\ind_c(D_E):=\partial([D_E])\in K_0(C^*(M;A))$.
  Because of homotopy invariance of K-theory, $[D_E]$ as well as $\ind_c(D_E)$
  does not depend on the choices made.

  If the dimension of $M$ is odd, $[(\chi(D_E)-1)/2] \in D^*(M;A)/C^*(M;A)$ is
  a projector and therefore represents a fundamental class $[D_E] \in
  K_0(M;A)$. 
  In this case the large scale index is defined by 
  $\ind_c(D_E):= \partial([D_E])\in K_1(C^*(M;A))$.
\end{definition}

\begin{remark}\label{rem:MF}
Let $M$ be a compact, closed, even dimensional spin manifold with Dirac
operator $D_E$ acting on  
$L^2(S\otimes E)$. Let $\mathbf B$ and $\mathbf K$ denote respectively the
space of all bounded 
and all compact operators on a separable Hilbert space $\mathbb H$. It is
clear that $C^*(M)=\mathbf K$.  
Therefore the coarse index $\ind_c (D_E)$ being an element in $K_0(\mathbf
K)=\Z$ is an integer. This is the usual Fredholm index of
$D_E$, compare \cite[Example after Definition 3.7]{Roe-index}.
\end{remark}

We recall the following well known vanishing result (compare
\cite{Roe-index,HankePapeSchick}). 
\begin{proposition}
  Let $M$ be a complete Riemannian spin manifold and $E$ a Hilbert $A$-module 
  bundle on $M$ as above.
  If $0$ is not in the spectrum of $D_E$ (in particular by the
  Schr\"odinger-Lichnerowicz formula  if $M$ has scalar
  curvature $\ge C>0$ and  $E$ is flat), then 
  \begin{equation*}
    \ind_c(D_E) = 0 \in K_{\dim M}(C^*(M;A)).
  \end{equation*}
\end{proposition}
\begin{proof}
  In this case, we can use for $\chi$ a function such that $\chi^2=1$ on the
  spectrum of $D_E$. Consequently, the formula defining $[D_E]$ shows that
  this class canonically lifts to a class $\rho(D_E)\in K_{\dim
    M+1}(D^*(M;A))$. Because of the exactness 
  of the K-theory sequence
  \begin{equation*}
    K_*(D^*(M;A))\to K_{*-1}(M;A)\xrightarrow{\partial} K_{*-1}(C^*(M;A))
  \end{equation*}
  this implies that $\ind_c(D_E)= \partial([D_E])=0$.
\end{proof}

\begin{proposition}\label{prop:MV_ind}
  Given a complete Riemannian spin manifold $M$ of dimension $n$ with
  Hermitean bundle 
  $E\to M$, write $E$ also for the pullback to $M\times \reals$. The
  Mayer-Vietoris isomorphism $\delta_{MV}$ of Corollary
  \ref{corol:KRq} sends $[D_{M\times \reals,E}]$ to $[D_{M,E}]$ and
  consequently also $\ind_c(D_{M\times\reals,E})$ to $\ind_c(D_{M,E})$:
  \begin{equation*}
    \begin{CD}
      K_{n+1}(M\times\reals) @>{\delta_{MV}}>{\iso}> K_n(M); &&
      [D_{M\times\reals}] \mapsto [D_{M,E}]\\
     @VV{\partial}V @VV{\partial}V  @VVV\\
     K_{n+1}(C^*(M\times\reals)) @>{\delta_{MV}}>{\iso}> K_n(C^*(M)); &&
    \quad \ind_c(D_{M\times\reals}) \mapsto \ind_c(D)
    \end{CD}
  \end{equation*}
\end{proposition}
\begin{proof}
  This crucial property of the Dirac operator is based on the principle that
  ``boundary of Dirac is Dirac''. The proof is given in \cite[Lemma
  4.6]{P.Siegel} and is based
  on the precise meaning of ``boundary of Dirac is Dirac'' as treated in
  \cite[Chapter 11]{HigsonRoe-K}.
\end{proof}

\section{Multi-partitioned manifolds and their large scale index}
\label{sec:largescaleind}
Throughout this section, assume that $M$ is a complete Riemannian manifold
which is multi-partitioned by the 
separating hypersurfaces $M_1,\dots,M_q$. Recall that this means in particular
that the latter are coarsely transversal in the
sense of Definition \ref{def:multipart} and near their common intersection
$N:= \bigcap_{k=1}^q M_k$ the mutual intersections are transversal in the usual
sense, and such that finally $N$ is compact.
We now prove Lemma \ref{lem:map_to_Rq}. 

\begin{definition}
  We write $M=M_k^+\cup M_k^-$ for the decomposition of $M$ induced by the
  hypersurface $M_k$. Define $h_k\colon M\to \reals$ as the signed distance to
  $M_k$, i.e.~$h_k(x)=d(x,M_k)$ if $x\in M_k^+$ and $h_k(x)=-d(x,M_k)$ if
  $x\in M_k^-$. Set $f\colon M\to \reals^q; x\mapsto
(h_1(x),\dots,h_q(x))$.
\end{definition}

\begin{proof}[Proof of Lemma \ref{lem:map_to_Rq}]
  By the triangle inequality, we have
  $\abs{d(x,X)-d(y,X)}\le d(x,y)$ for an arbitrary subset $X\subset M$. Moreover, as $M$ has a length metric and
  $M_k$ is separating, $x\in M_k^+$ and $y\in M^-_k$ satisfy
  $d(x,M_k)+d(y,M_k)\le d(x,y)$. It follows that $h_k\colon M\to\reals$ is a
  $1$-Lipschitz map, therefore $f$ is  $\sqrt{q}$-Lipschitz. The
  condition that $N$ is compact and that the $M_k$ are coarsely transversal
  implies 
  that the inverse image of every bounded subset of $\reals^k$ under $f$ is
  bounded.  This finishes the proof of Lemma \ref{lem:map_to_Rq}.
\end{proof}
We have now explained all the ingredients for the statement of Theorem
\ref{theo:part_mf}. Indeed, in Section \ref{section2} we essentially already
proved the model case of this theorem, which reads as follows:

\begin{lemma}\label{lem:model_part_mf}
  Let $N$ be a compact $n-q$-dimensional spin manifold and $E\to N$ a Hermitean 
  bundle. Write $E$ also for the pullback to $N\times \reals^q$ along the
  projection $f\colon 
  N\times \reals^q\to \reals^q$. 
  For this special case, the assertion of Theorem \ref{theo:part_mf} holds.
\end{lemma}
\begin{proof}
  By naturality of the Mayer-Vietoris sequence, the following
 diagram is commutative:
  \begin{equation*}
  \begin{CD}
     K_{n}(C^*(N\times \reals^q)) @>{f_*}>> K_{n}(C^*(\reals^q))\\
     @V{\iso}V{\delta_{MV}^q}V  @V{\iso}V{\delta_{MV}^q}V\\
    K_{n-q}(C^*(N)) @>{pr_*}>{\iso}> K_{n-q}(C^*(\reals^0))\\
      @VV{\iso}V @VV{\iso}V \\
    K_{n-q}(\complexs) @>=>> K_{n-q}(\complexs)
    \end{CD}
  \end{equation*}
   By definition, the right vertical composition is $\kappa$ so that
   $\ind_p(D_{N\times\reals^q,E})\in \Z$ is the image of
   $\ind_c(D_{N\times\reals^q,E})$ under the map to the right lower corner.
   However, by Proposition \ref{prop:MV_ind},
   \begin{equation*}
   \delta_{MV}^q(\ind_c(D_{N\times\reals^q,E}))= \ind_c(D_{N,E})\in
   K_{n-q}(C^*(N)),
 \end{equation*}
and the latter is mapped to $\ind(D_{N,E})\in K_{n-q}(\complexs)$
   under 
   the isomorphism $K_{n-q}(C^*(N))\to K_{n-q}(\C)$ by Remark
   \ref{rem:MF}. Note that $K_{n-q}(\C)=\Z$ if $n-q$ is even, and
   $K_{n-q}(\C)=\{0\}$ if $n-q$ is odd in which case the last statement is empty.
\end{proof}

The main novelty of this note is the localization result for the partitioned
manifold index. It follows from the following localization result for the
K-theory of $C^*(\reals^n)$.

\begin{definition}
  Two operators $T_1,T_2\in D^*(\reals^q;A)$ are said to coincide on an open
  set $U\subset \reals^q$ if and only if $T_1s=T_2s$ for all $s$ with
  $\supp(s)\subset U$. 
\end{definition}

\begin{proposition}\label{prop:loca}
  Let $T_1,T_2\in D^*(\reals^q)$ be two operators which coincide on a
  non-empty open set $U\subset\reals^q$. Assume that $[T_1]$ and $[T_2]$
  represent elements in $K_j(D^*(\reals^q;A)/C^*(\reals^q;A))$, i.e.~are
  either idempotents (for $j$ even) or invertible (for $j$ odd) modulo
  $C^*(\reals^q;A)$.

  Just from the fact that $T_1$ and $T_2$ coincide on $U$, it then follows
  that
  \begin{equation*}
[T_1]=[T_2] \in K_j(D^*(\reals^q;A)/C^*(\reals^q;A)).
\end{equation*}

\end{proposition}
\begin{proof}
  By translation invariance of $K_*(D^*(\reals^q;A))$ and making $U$ smaller,
  if necessary, we can assume that $U=B_r(0)$ for some $r>0$.

  We use the auxiliary space $\reals^q\setminus U$. We apply the
  Mayer-Vietoris principle to the decomposition 
  \begin{equation*}
\reals^q\setminus B_r(0) =
  (\reals^q_-\setminus B_r(0) )\cup (\reals^q_+\setminus B_r(0))
\end{equation*}
  The intersection is $\reals^{q-1}\setminus B_r(0)$ and the half spaces are
  flasque. For $q>1$ the decomposition is coarsely excisive and, using
  Proposition \ref{prop:flasque_vanishing} and Theorem \ref{theo:MV} we get a
  Mayer-Vietoris isomorphism
  \begin{equation*}
    K_*(\reals^q\setminus B_r(0);A)\to K_{*-1}(\reals^{q-1}\setminus B_r(0);A).
  \end{equation*}

  Finally, in the case $q=1$, write $\reals':=(-\infty,-r]\cup
  [r,\infty)\subset\reals$.  Recall that for arbitrary metric space $X$ and
  $C^*$-algebra $A$ every element $T\in D^*(X;A)$ can be written as a sum of
  an operator $T_1\in D^*(X;A)$ of arbitrarily small propagation and $T_2\in
  C^*(X;A)$. Using the notation of  
  Definition \ref{def:rel_algs}, therefore $D^*(\reals';A)/C^*(\reals';A)$
  decomposes as a direct sum 
  \begin{equation*}
 (D^*((-\infty,-r]\subset
             \reals';A)+C^*(\reals';A))/C^*(\reals';A) \oplus
             (D^*([r,\infty);A)+C^*(\reals';A))/C^*(\reals';A).
           \end{equation*}
By the Noether isomorphism theorems for the two summands we have
\begin{gather*}
 (D^*((-\infty,-r]\subset \reals';A)+C^*(\reals';A))/C^*(\reals';A) \iso
             D^*((-\infty,-r];A)/C^*((-\infty,-r];A),\\
 (D^*([r,\infty)\subset \reals';A)+C^*(\reals';A))/C^*(\reals';A)\iso
             D^*([r,\infty);A)/C^*([r,\infty);A).
\end{gather*}

  Therefore 
    \begin{equation*}
      K_*(\reals';A) = K_{*}((-\infty,-r];A) \oplus K_*([r,\infty);A)= 0
    \end{equation*}
  again using that the half line is flasque. 
(As we are computing locally finite K-homology here, the Mayer-Vietoris
sequence holds for 
general, not necessarily coarsely excisive, decompositions which would
immediately implies the statement. We avoid the use of it because we did not
establish this general Mayer-Vietoris sequences for locally finite K-homology
with coefficients in the $C^*$-algebra $A$.)

  By assumption, $T_1,T_2$ coincide on $U$. We claim that this implies that
  $T_1-T_2\in D^*(Y\subset \reals^q;A)$ with $Y=\reals^q\setminus U$. The support condition is
  automatic, as $U_r(Y)=\reals^q$. If $\phi\colon\reals\to\complexs$ has
  support on $U$ then $(T_1-T_2)\phi=0$ by
  assumption, therefore $(T_1-T_2)\phi$ is compact. Because $T_1,T_2$ are in
  $D^*(\reals^q;A)$, 
  the commutator $[\phi,T_1-T_2]$ is compact. This gives then the required
  remaining compactness of $\phi(T_1-T_2)$.

  From this, we conclude that the images of $T_1$ and $T_2$ in
  $D^*(\reals^q;A)/(D^*(Y\subset\reals^q;A)+C^*(\reals^q;A))$ coincide.

  By Proposition \ref{prop:rel_K},
  \begin{equation*}
    K_*(D^*(Y\subset \reals^q;A)/C^*(Y\subset \reals^q;A)) \iso K_{*-1}(Y;A) = 0.
  \end{equation*}
  Therefore, the long exact K-theory sequence of the extension
  \begin{multline*}
    0\to D^*(Y\subset\reals^q;A)/C^*(Y\subset\reals^q;A) \to
    D^*(\reals^q;A)/C^*(\reals^q;A)\\
\to
    D^*(\reals^q;A)/(D^*(Y\subset\reals^q;A)+C^*(\reals^q;A))\to 0
  \end{multline*}
  gives the isomorphism, induced by the projection
  \begin{equation*}
    K_*(\reals^q;A)\xrightarrow{\iso}
    K_{*+1}(D^*(\reals^q;A)/(D^*(Y\subset\reals^q;A)+C^*(\reals^q;A))).
  \end{equation*}
  We observed above that the images of $T_1$ and $T_2$ in the right hand
  algebra coincide. Because of the isomorphism, $[T_1]=[T_2]\in
  K_*(\reals^q;A)$, as we had to prove.
\end{proof}

The localization theorem, Proposition \ref{prop:loc_ind}, now is a rather
direct corollary, as we want to prove next. Assume therefore the situation of
Proposition \ref{prop:loc_ind}, with two manifolds $f\colon M\to \reals^q$,
$f'\colon M'\to \reals^q$ together with Hilbert $A$-module bundles with
connection which are locally isometric on open neighborhoods 
$U,U'$ of the inverse images $N,N'$ of $0$ via isometries $\psi,\Psi$. As 
$f$ is proper and continuous and $U$ is an open neighborhood of
$N$ (and the corresponding situation for $M'$), if
we choose $t>0$ sufficiently small then $f^{-1}(B_t(0))\subset U$ and
$(f')^{-1}(B_t(0))\subset U'$. Choose $r>0$ such that
$U_r(f^{-1}(B_{t/2}(0)))\subset f^{-1}(B_t(0))$. Because $U,U'$ are isometric,
the same is then also true for $M'$. Next, choose a smooth chopping function
$\chi$ as for
the definition of $[D_E]$ such that its Fourier transform $\hat\chi$ (which is
a distribution which is smooth outside $0$) has support in $(-r/4,r/4)$. By
the Fourier inversion formula and unit propagation speed of the wave operator
(which implies that $\chi(D_E),  
\chi(D_{E'})$ have propagation $r/4$), 
$\chi(D_E)s= \Psi^{-1}\chi(D_{E'})\Psi s$ for each $s$ with support in
$f^{-1}(B_{t/2}(0))$. Next, for the construction of $f_*\colon D^*(M;A)\to
D^*(\reals^q;A)$ and $f'_*\colon D^*(M';A)\to D^*(\reals^q;A)$ choose
isometries $V,V'$ as in Definition \ref{def:functor} with propagation smaller
than $r/4$. These isometries can be constructed locally and patched
together. We can
therefore in addition arrange that
\begin{equation}\label{eq:supp}
Vs=V'\Psi s \quad\text{if} \qquad \supp(s)\subset
U_{3r/4}(f^{-1}(B_{t/2}(0)))\subset U.
\end{equation}
As
$\innerprod{V^*u,s}_{L^2(S\tensor E)}=\innerprod{u,Vs}_{L^2(\reals^q)}$, 
the fact that $V$ has propagation $r/4$
implies that if $\supp(u)\subset B_{t/2}(0)$ then the support of $V^*u$ is
contained in $U_{r/4}(f^{-1}(B_{t/2}(0)))$. Then Equation \eqref{eq:supp}
implies that $\Psi V^* u = (V')^* u$ for these $u$.

Taken together, we get that $V\chi(D_E)V^* u =
V'\chi(D_{E'})(V')^*u$ if 
$u$ is supported on $B_{t/2}(0)$. This implies that $f_*((\chi(D_E)+1)/2)$ and
$f'_*((\chi(D_{E'})+1)/2)$ coincide on $B_{t/2}(0)\subset\reals^q$. If $M$ has
even 
dimension, in addition we have to choose the measurable bundle isomorphisms $V_S,V'_{S'}$
between the positive and negative spinor bundles. Again, this construction is
local and we can therefore arrange that for sections supported on $U$ the
isometry $\Psi$ intertwines these bundle isomorphisms, i.e.~$\Psi V_Ss =
V'_{S'}\Psi s$ if $\supp(s)\subset U$.

It then follows that, if $u$ is supported on $B_{t/2}(0)\subset \reals^q$ then
\begin{equation*}
f_*(V_S^*\chi(D_E)_+)u\stackrel{\text{Def}}{=} V (V_S^*\chi(D_E)_+)V^* u= V'
(V'_{S'}\chi(D_{E'})_+)(V')^* u = f'_*(V'_{S'}\chi(D_{E'})_+)u.
\end{equation*}

To summarize: in all dimensions the representatives classes $f_*[D_E]$ and
$f'_*([D_{E'}'])$ 
coincide on the non-empty open subset $B_{t/2}(0)\subset\reals^q$. By
Proposition \ref{prop:loca}, $f_*([D_E])=f'_*([D_{E'}])\in K_{\dim
  M}(\reals^q;A)$. By naturality of the boundary map of the K-theory
long exact then also
\begin{equation*}
  f_*(\ind_c(D_E)) = f'_*(\ind_c(D_{E'}))\in K_{\dim M}(C^*(\reals^q;A)),
\end{equation*}
as we have to prove.

Finally, we now can give the proof of the multi-partitioned
manifold index theorem \ref{theo:part_mf}.

Given $f\colon M\to \reals^q$ as in Theorem \ref{theo:part_mf}, by homotopy
invariance of the index we can deform the metric on $M$ and connection on $E$
in a
neighborhood $U$ of $N$ such that it is isometric to a neighborhood of
$N\times \{0\}$ in $N\times \reals^q$ (with product structure) without changing
$\ind(D_{M,E})$.

By Proposition \ref{prop:loc_ind}, which we just
proved, $f_*(\ind_c(D_{M,E}))=f_*(\ind_c(D_{N\times\reals^q, E|_N}))$. But for
the latter we already proved in Lemma \ref{lem:model_part_mf} that
$\kappa(\ind_c(D_{N\times\reals^q,E|_N}))= \ind(D_{N,E|_N})\in \Z$. Therefore
Theorem \ref{theo:part_mf} is established.


We next apply the multi-partitioned manifold index theorem  to prove 
 non-existence theorems for metrics with positive scalar curvature.

\begin{lem}\label{vanlem}
 Let $M$ be a spin manifold with spinor bundle $S$ and let $E$ be a flat
 Hilbert $A$-module bundle  on $M$. 
 Let $D_E$ be a Dirac operator twisted by $E$.
 \begin{enumerate}
 \item If there is a constant $C>0$ such that the scalar curvature of $g$ is grater than 
 $C$ outside $Y$, then $\ind_c(D_E)$ is in the image of $K_*(C^*(Y\subset
 M;A))\to K_*(C^*(M;A))$. 
 \item Let $(M',g')$ be another complete manifold.
 If $f\colon M\to M'$ is a proper and uniformly expanding map with
 $f(Y)\subset Y'\subset M'$ then 
 $f_*(\ind_c D)\in K_*(C^*(M';A))$ 
 takes its value in the image  of 
 $K_*(C^*(Y'\subset M';A))\to K_*(C^*(M';A))$.
\end{enumerate}
\end{lem}
\begin{proof}
  The first part for the case $A=\C$ is in \cite[page 22]{Roe-index} without
proof.  
In \cite{RoeFriedrich} a proof of this special case is
provided using the Friedrichs extension of symmetric operators which are
bounded below.  A sketch of a proof using similar ideas for general $A$ is
given in \cite{Zadeh2}. Simultaneously, a detailed proof using Fourier
inversion techniques was given in the G\"ottingen thesis of Daniel Pape
\cite{Pape} and appeared in \cite{HankePapeSchick}. 

 The
  second part is a direct consequence of the first part and of naturality.
\end{proof}

With this Lemma, we are in the position to prove Theorem
\ref{theo:obstructmet}. Assume therefore that $f\colon M\to \reals^q$ is a
proper and uniformly expansive map defined on a complete spin manifold $M$ and
assume that $f$ is smooth near $N:=f^{-1}(0)$ such that $0$ is a regular
value. Let $D$ stands for the Dirac operator of $M$.  
If the scalar curvature of $M$ is uniformly positive on the quadrant 
$\bigcap_{k=1}^q M_k^+ = f^{-1}([0,\infty)^q)$, then by Lemma \ref{vanlem}
$f_*(\ind_c(D))$ lies in the image of $K_*(C^*(\reals^q\setminus
[0,\infty)^q\subset \reals^q))\to K_*(C^*(\reals^q)$. Now,
$\reals^q\setminus [0,\infty)^q$ is a flasque space, therefore by Proposition
\ref{prop:flasque_vanishing} $K_*(C^*(\reals^p\setminus
[0,\infty)^q))=K_*(C^*(\reals^q\setminus [0,\infty)^q\subset \reals^q))$
vanishes. 
It follows, under the positivity assumption on the scalar curvature, that
$f_*(\ind_c(D))=0$. 
On the other hand, by Theorem \ref{theo:part_mf} $\kappa
f_*(\ind_c(D))=\ind(D_N)=\hat A(N)\neq 0$. This contradiction proves the
theorem.

\section{Generalization to arbitrary coefficient $C^*$-algebras and open questions}
\label{sec:generalization}

The partitioned manifold index theorem ($q=1$) for Hilbert
$A$-module bundle coefficients is established in 
\cite{Zadeh1} and plays an important role in \cite{HankePapeSchick} in the
proof of the codimension $2$-obstruction to positive scalar
curvature. Generalizations of this obstruction to higher codimensions, which
would be very interesting, most likely would require to generalize Theorem
\ref{theo:part_mf} to arbitrary $C^*$-algebras $A$.
The most natural
example is $A=C^*(\pi_1(M))$, the (reduced or maximal) group $C^*$-algebra, and
$E=\tilde M\times_{\pi_1M} C^*(\pi_1(M)$), the Mishchenko line bundle.

Our proof of Theorem \ref{theo:part_mf} generalizes, provided the following
input is established:
\begin{enumerate}
\item the generalization of Remark \ref{rem:MF}, namely the identification of
  the coarse index of a compact space with the 
  Mishchenko-Fomenko index (or any other version classically used) ---this is
  folk knowledge, but its hard to find suitable references;
\item the product formula for the $A$-module twisted index generalizing
  Proposition \ref{prop:MV_ind}. Actually, one should establish a more general
  product formula for coarse indices as follows. We formulate it as a
  conjecture but strongly believe that it is true and that the method of proof
  of the current paper can be generalized to give a proof of the statement.
  \begin{conjecture}
    Given complete Riemannian manifolds $M_1$ with Hilbert $A_1$-module bundle  
  $E_1\to M_1$ and $M_2$ with Hilbert $A_2$-module bundle $E_2\to M_2$, the
  external tensor product of operators on $L^2(M_1, S_1\tensor
  E_1)\tensor L^2(M_2,S_2\tensor E_2)\xrightarrow{\iso}
  L^2(M_1\times M_2, S_1\tensor S_2\tensor E_1\otimes E_2)$ defines an
  algebra homomorphism $C^*(M_1;A_1)\tensor C^*(M_2;A_2)\to C^*(M_1\times
  M_2;A_1\tensor A_2) $ (which in general is
  not an isomorphism) and an induced map in K-theory
  \begin{equation*}
    \alpha\colon K_i(C^*(M_1;A_1))\tensor K_j(C^*(M_2;A_2))\to
    K_{i+j}(C^*(M_1\times M_2;A_1\tensor A_2)).
  \end{equation*}

  In this situation, we should have a product formula for the index of the
  Dirac operator: 
    If $M_1, M_2$ (and therefore also $M_1\times M_2$) are spin
    manifolds then for the coarse indices we obtain
    \begin{equation*}
      \alpha( \ind_c(D_{M_1,E_1})\tensor \ind_c(D_{M_2,E_2})) =
      \ind_c(D_{M_1\times M_2, E_1\tensor E_2}) \in K_{\dim(M_1)+\dim(M_2)}
      (C^*(M_1\times M_2;A_1\tensor A_2)).
    \end{equation*}
  \end{conjecture}
  Indeed, this is the most non-trivial generalization as this index formula at
  some point requires explicit calculations. The best route to prove
  this is probably to give a new description of the coarse index, e.g.~using
  the picture of \cite{Guentner-Higson} or Kasparov theory, which is more
  directly adapted to such product 
  situations. Of course, then the task remains to identify the new definition
  of the index with the old one.
This should follow from the methods of Rudolf Zeidler's \cite{Zeidler}.
\end{enumerate}

It would perhaps also be desirable to give a written account of the details of
the generalizations of \ref{prop:flasque_vanishing}, \ref{prop:rel_K},
\ref{theo:MV}, on the other hand this is really straightforward.

\begin{remark}
  Note that the proof of Theorem \ref{theo:obstructmet} works without any
  change replacing $\complexs$ by a
  general $C^*$-algebra $A$, as soon as Theorem \ref{theo:part_mf} is
  generalized.
\end{remark}

 \begin{question}
   \begin{enumerate}
   \item It would be interesting to
 work out explicit example situations where our theorem, in particular the
  obstruction to positive scalar curvature, can be applied.
\item In particular, can one establish a higher version of the codimension
  $2$-obstruction to positive scalar curvature, using e.g.~intersections of
  several transversal codimension $2$-submanifolds, or an iterative procedure
  (with a second codimension $2$-submanifold inside the first,\ldots).

  Very speculatively, can one perhaps pass to even higher codimension
  submanifold obstructions?
   \item In \cite{Piazza-Schick}, the secondary invariants $\rho(g)\in
     K_{*+1}(D^*M)$ 
are studied and a secondary partitioned manifold index theorem for them is
established. It would be interesting to see whether the localization methods
of this paper carry over and whether one can 
establish a secondary multi-partitioned index theorem along these lines. For
the product formula for the secondary index, compare \cite{Zeidler}.
   \end{enumerate}
 \end{question}








\begin{thebibliography}{10}


\bibitem{GrLa3}
Mikhail Gromov and H.~Blaine Lawson, Jr.
\newblock Positive scalar curvature and the {D}irac operator on complete
  {R}iemannian manifolds.
\newblock {\em Inst. Hautes \'Etudes Sci. Publ. Math.}, (58):83--196 (1984),
  1983.

\bibitem{HaSc1}
Bernhard~Hanke and Thomas~Schick.
\newblock Enlargeability and index theory.
\newblock {\em J. Differential Geom.}, 74(2):293--320, 2006.


\bibitem{HaSc2}
Bernhard Hanke and Thomas Schick.
\newblock Enlargeability and index theory: Infinite covers.
\newblock {\em $K$-Theory}, 38(1):23--33, 2007.

\bibitem{HankePapeSchick}
 Bernhard Hanke, Daniel Pape,  and Thomas Schick.
 \newblock {Codimension two index obstructions to positive scalar curvature}.
  \newblock  To appear in {\em Annales  de l'Institut Fourier}, 
  preprint version arXiv:1403.4094

\bibitem{Higson-cobordism}
Nigel Higson.
\newblock A note on the cobordism invariance of the index.
\newblock {\em Topology}, 30(3):439--443, 1991.

\bibitem{Guentner-Higson}
Nigel Higson and Erik Guentner.
\newblock Group $C^*$-algebras and K-theory.
\newblock In: \emph{Noncommutative geometry}.
\newblock volume 1831 of Lecture Notes in Math., pages 137–251
\newblock Springer, Berlin, 2004.


\bibitem{HigsonRoeCBC}
Nigel Higson and John Roe.
\newblock On the coarse Baum-Connes conjecture. 
\newblock In \emph{Novikov conjectures, index
theorems and rigidity}, Vol. 2 (Oberwolfach, 1993).
\newblock Volume 227 of London Math. Soc. Lecture Note Ser.,
pages 227--254.
\newblock Cambridge Univ. Press, Cambridge, 1995.

\bibitem{HigsonRoe-K}
Nigel Higson and John Roe.
\newblock {Analytic K-homology}.
\newblock {Oxford University Press}, 2000.

\bibitem{HigsonRoeYu}
Nigel Higson, John Roe, and Guoliang Yu.
\newblock A coarse {M}ayer-{V}ietoris principle.
\newblock {\em Math. Proc. Cambridge Philos. Soc.}, 114(1):85--97, 1993.

\bibitem{Kucer}
Dan Kucerovsky.
\newblock Functional calculus and representations of {$C_0(\mathbb C)$} on a
  {H}ilbert module.
\newblock {\em Q. J. Math.}, 53(4):467--477, 2002.

\bibitem{Lance}
Christopher E.~Lance
\newblock Hilbert $C^*$-modules.
\newblock volume 210 of \emph{London Mathematical Society Lecture Note Series}.
\newblock Cambridge University Press, 1995.

\bibitem{FoMi}
Alexander~S. Mi{\v{s}}{\v{c}}enko and A.~T. Fomenko.
\newblock The index of elliptic operators over {$C^{\ast} $}-algebras.
\newblock {\em Izv. Akad. Nauk SSSR Ser. Mat.}, 43(4):831--859, 967, 1979.

\bibitem{Pape}
  Daniel Pape.
\newblock Index theory and positive scalar curvature.
\newblock Dissertation, Mathematisches Institut, Georg-August-Universit\"at
G\"ottingen, 2011.

\bibitem{Piazza-Schick}
Paolo Piazza and Thomas Schick.
\newblock Rho-classes, index theory and Stolz’ positive scalar curvature
sequence.
\newblock {\em J.~Topol.}, 7(4):965–1004 (2014). doi:
10.1112/jtopol/jtt048



\bibitem{Roe-partitioning}
John Roe.
\newblock Partitioning noncompact manifolds and the dual {T}oeplitz problem.
\newblock In {\em Operator algebras and applications, {V}ol.\ 1}, volume 135 of
  {\em London Math. Soc. Lecture Note Ser.}, pages 187--228. Cambridge Univ.
  Press, Cambridge, 1988.


\bibitem{Roe-index}
John Roe.
\newblock {\em Index theory, coarse geometry, and topology of manifolds},
  volume~90 of {\em CBMS Regional Conference Series in Mathematics}.
\newblock Published for the Conference Board of the Mathematical Sciences,
  Washington, DC, 1996.

\bibitem{RoeFriedrich}
    John Roe.
\newblock   Positive curvature, partial vanishing theorems, and coarse indices.
\newblock Preprint,  arXiv:1210.6100, 2012.

\bibitem{P.Siegel}
Paul~{Siegel}.
\newblock {The Mayer-Vietoris Sequence for the Analytic Structure Group}.
\newblock Preprint, arXiv:1212.0241,  2012.

\bibitem{Zadeh1}
Mostafa~Esfahani Zadeh.
\newblock Index theory and partitioning by enlargeable hypersurfaces.
\newblock {\em J. Noncommut. Geom.}, 4(3):459--473, 2010.

\bibitem{Zadeh2}
Mostafa~Esfahani Zadeh.
\newblock A note on some classical results of {G}romov-{L}awson.
\newblock {\em Proc. Amer. Math. Soc.}, 140(10):3663--3672, 2012.

\bibitem{Zeidler}
Rudolf Zeidler.
\newblock Positive scalar curvature and product formulas for secondary index
invariants.
\newblock Preprint, arXiv:1412.0685.

\end{thebibliography}

\end{document}